\documentclass[amssymb,amsfonts,refcheck,12pt,verbatim,righttag]{amsart}
\usepackage{amssymb}
\usepackage{graphicx}

\usepackage{graphics,color}
\usepackage{color}

\usepackage{graphicx,color,tikz,caption,subcaption}

\numberwithin{equation}{section} % \renewcommand{\rm}{\normalshape} %
\theoremstyle{plain}

\newtheorem{thm}{Theorem}[section]
\newtheorem{lem}[thm]{Lemma}
\newtheorem{pro}[thm]{Proposition}

\newtheorem{de}[thm]{Definition}
\newtheorem{rem}[thm]{Remark}

\def\R {{\Bbb R}}
\def\N {{\Bbb N}}
\def\Z {{\Bbb Z}}

\begin{document}
\baselineskip 14pt
\title{On  multifractal formalism for self-similar measures with overlaps.}

\author{Julien Barral}
\address{Laboratoire de G\'eom\'etrie, Analyse et Applications, CNRS, UMR 7539, Universit\'e Sorbonne Paris Nord, CNRS, UMR 7539,  F-93430, Villetaneuse, France}
\email{barral@math.univ-paris13.fr}

\author{De-Jun Feng}
\address{Department of Mathematics\\ The Chinese University of Hong Kong\\ Shatin,  Hong Kong\\ }
\email{djfeng@math.cuhk.edu.hk}

%\keywords{Hausdorff dimension, exact dimensionality, invariant measures, iterated function systems, self-affine sets}

\date{}

\begin{abstract}
Let $\mu$ be a self-similar measure generated by an IFS $\Phi=\{\phi_i\}_{i=1}^\ell$ of similarities on $\mathbb R^d$ ($d\ge 1$).  When $\Phi$ is dimensional regular (see Definition~\ref{de-1.1}),  we give an explicit  formula for the $L^q$-spectrum $\tau_\mu(q)$ of $\mu$ over $[0,1]$, and show that  $\tau_\mu$  is differentiable over $(0,1]$ and the multifractal formalism holds for $\mu$ at any $\alpha\in [\tau_\mu'(1),\tau_\mu'(0+)]$. We also  verify the validity of the multifractal formalism of $\mu$  over $[\tau_\mu'(\infty),\tau_\mu'(0+)]$ for   two new classes of  overlapping algebraic IFSs by showing that the asymptotically weak separation condition holds. For one of them, the proof appeals to the recent result of Shmerkin \cite{Shmerkin2019} on the $L^q$-spectrum of  self-similar measures.
\end{abstract}

\maketitle

\section{Introduction}\label{S-1}

Self-similar sets and measures are natural and important objects in fractal geometry, at the interface of geometric measure theory, ergodic theory, number theory and harmonic analysis (see the recent surveys \cite{HochmanJMD,Hochman2018, Shmerkin2019b}). Spectacular and influential advances in the dimension theory  of these objects have been achieved  in the recent period \cite{Furstenberg2008,Hochman-S2012,Hochman2014,Hochman-S2015, BreuillardVarju2015, Shmerkin2019, Varju2019, Wu2019}, in particular in connection with the resolutions of  Furstenberg's conjectures  on the Hausdorff dimension of  the sums and intersections of $\times 2$- and $\times 3$-invariant sets on the 1-dimensional torus.

In this paper, we present  some new results on the multifractal analysis of self-similar measures. One  of them concerns the  precise value of their $L^q$-spectra over $[0,1]$ and the validity of the multifractal formalism, while the other ones provide sufficient conditions under which the underlying iterated function system (IFS) satisfies the {\it asymptotically weak separation condition} (AWSC), see Definition~\ref{AWSC}.   (This separation condition guarantees the validity of the multifractal formalism in the range of $q>0$ \cite{Feng2012}.) Most of these results rely on the achievements in \cite{BarralFeng2013}, \cite{Hochman2014} and \cite{Shmerkin2019}.  Before giving the backgrounds and  precise formulations of our results,  below we first introduce some necessary notation and definitions.

Recall that for a finite  Borel measure $\eta$ on $\R^d$ with compact support,  the {\it $L^q$-spectrum} of $\eta$ is defined as
$$
\tau_\eta(q)=\liminf_{r\rightarrow 0}\frac{\log \Theta_\eta(q, r)}{\log r},
$$
where
\begin{equation}
\label{e-tz1}
\Theta_\eta(q, r)= \sup
\sum_{i}\eta (B(x_i,r))^q ,\qquad r>0,\; q\in \R,
\end{equation}
and the supremum is taken over all  families of disjoint closed balls $%
\{B(x_i,r)\}_{i}$  of radii $r$ with centers $x_{i}\in \mbox{supp}(\eta)$.
It is easily checked that ${\tau}_\eta(q)$ is a concave
function of $q$ over $\R$. For $x\in \R^d$, the {\it local dimension} of $\eta$ at $x$ is defined as
$$
d(\eta,x)=\lim_{r\to 0}\frac{\log \eta(B(x,r))}{\log r},
$$
provided that the limit exists. Otherwise we use $\overline{d}(\eta,x)$ and $\underline{d}(\eta,x)$ to denote the upper and lower limits, respectively.
For $\alpha\in \R$, denote
$$
E_\eta(\alpha)=\left\{x\in \R:\; d(\eta,x)=\alpha\right\}.
$$
 We say that  {\it the multifractal formalism  holds for $\eta$  at $\alpha$} if $$
\dim_H E_\eta(\alpha)=\tau_\eta^*(\alpha):=\inf\{\alpha q-\tau_\eta(q):q\in\mathbb R\},$$ with the convention that $\dim_H \emptyset=-\infty$, where $\dim_H$ stands for the Hausdorff dimension  (see e.g. \cite{Falconer2003, Mattila1995} for the definition). One of the main objectives of multifractal analysis is to study the validity of this formalism for natural fractal  measures. For backgrounds and the rigorous mathematical foundations of the multifractal formalism, we refer to \cite{Falconer2003, Olsen1995, Pesin1997}.

From now on, suppose that $\mu=\mu_{\Phi,{\bf p}}$ is the  self-similar  measure associated with an IFS $\Phi=\{\phi_i\}_{i=1}^\ell$ of similarities on $\R^d$ and a probability vector ${\bf p}=(p_1,\cdots,p_\ell )$ with strictly positive entries.   That is, $\mu$ is the unique Borel probability measure on $\R^d$ satisfying the following relation:
$$
\mu=\sum_{i=1}^\ell p_i\mu\circ\phi_i^{-1}.
$$
 (See Section~2 or \cite{Falconer2003, Hutchinson1981} for more details.)  The measure $\mu$ is fully supported on the attractor $K$ of $\Phi$. It is known  that $\mu$ is exact dimensional in the sense that $d(\mu,x)$ equals a constant for $\mu$-a.e.~$x$ \cite{FengHu2009}. We use $\dim_H\mu$ to denote this constant  and call it the {\it Hausdorff dimension of $\mu$}.

Let us  briefly  summarize  some general results obtained up to now for such a measure concerning the validity of the multifractal formalism and the possible expressions for the function $\tau_\mu$.

The multifractal formalism is known to hold for $\mu$ at any $\alpha\in \mathbb R$ if $\{\phi_i\}_{i=1}^\ell$  satisfies the {\it open set condition} (see Section~\ref{S-2} for the definition), and moreover in this case, the $L^q$-spectrum  of $\mu$ is equal to the analytic function $T(q)$, which is the unique solution of the equation
\begin{equation}\label{Lqspec}
\sum_{i=1}^\ell p_i^qr_i^{-T(q)}=1,
\end{equation}
where $r_i$ stands for the contraction ratio of $\phi_i$; see  \cite{Patzschke2017, CawleyMauldin1992, LauNgai1999, Olsen1995}. It remains an interesting and challenging problem to study the case when the open set condition fails.

 It is known \cite{FengLau2009} that if $\{\phi_i\}_{i=1}^\ell$ satisfies the {\it weak separation condition} (WSC, see Section~\ref{S-2} for the definition)  and one considers the restriction of $\mu$ to some well chosen open ball, the multifractal formalism still holds at any $\alpha$.  If one relaxes this assumption to the AWSC (see Section~\ref{S-2}), then the multifractal formalism holds for $\mu$  at any $\alpha\in [\tau_\mu'(\infty), \tau_\mu'(0+)]$, where  $\tau_\mu'(\infty):=\lim_{q\to \infty}\tau_\mu(q)/q$ and $\tau_\mu'(0+)$ is the right derivative of $\tau_\mu$ at $0$ (see \cite[Theorem 1.3]{Feng2012}).  While without any separation condition, the multifractal formalism holds for $\mu$  at $\alpha=\tau_\mu'(q)$ provided that $\tau_\mu$ is differentiable at $q$ and $q\geq 1$ \cite{Feng2007}.  We emphasize that both  WSC and AWSC are weaker than the open set condition and  are satisfied by many interesting examples. For instance, consider an homogeneous IFS  $\{\rho x+a_i\}_{i=1}^\ell$ on $\R$ with $a_i\in \Z$. Then  this IFS satisfies the WSC if $1/\rho$ is a {\it Pisot number}, and the AWSC if $1/\rho$ is  a Pisot or {\it Salem number} (see e.g. \cite{LauNgai1999, Feng2007}). Recall that a Pisot number is an algebraic integer whose conjugates are all less than 1 in modulus, whilst a Salem number is an algebraic integer whose conjugates are all less than or equal to  1 in modulus, with at least one of which on the unit circle.
  The reader is referred to \cite{FengLau2009} for a large literature concerning concrete classes of self-similar measures with the WSC. It is worth pointing out that under the assumption of the WSC, the multifractal formalism may break down in the region of  $q<0$ (see e.g. \cite{HuLau2001, Shmerkin2005, Testud2006}).

 Thanks to the recent result obtained by Shmerkin~\cite{Shmerkin2019} (see Theorem~\ref{thm-S}), we know that when $d=1$, under the  {\it exponential separation condition} (ESC, see Definition~\ref{ESC}), the $L^q$-spectrum of $\mu$  is given by
  \begin{equation}
  \label{e-Shmerkin}
  \tau_\mu(q)=\min \{q-1,T(q)\}\; \mbox{  for all }q\ge 1,
  \end{equation}
   where $T$ is defined as in \eqref{Lqspec}, so $\tau_\mu$ is differentiable on $(1,\infty)$  except, perhaps at a single point $q_0>1$ with $q_0-1=T(q_0)$, whence the multifractal formalism holds for $\mu$  at every $$\alpha\in [\tau_\mu'(\infty), \tau_\mu'(1+)]\backslash (\tau_\mu'(q_0+), \tau_\mu'(q_0-));$$ see  Remark \ref{rem-A1}(ii).   We emphasize that the ESC is also weaker than the open set condition, but not like WSC and AWSC, the ESC does not allow exact overlaps. It is also worth noting that \eqref{e-Shmerkin} is also known to hold for almost all self-similar measures, for all $q\in [1,2]$ and $d\ge 1$, by a result of Falconer~\cite{Falconer1999}.

 The notion of ESC was introduced by Hochman \cite{Hochman2014,Hochman2019+}, who proved that under the assumption of the ESC on $\Phi$,   when $d=1$,  or $d>1$ but with additional irreducibility conditions on $\Phi$,  the dimension of $\mu$  is equal to $\min\{d,\dim_S\mu\}$, where
$$\dim_S \mu=T'(1)=\displaystyle\frac{\sum_{i=1}^\ell p_i\log p_i}{\sum_{i=1}^\ell p_i\log r_i},
$$
which is called the {\it similarity dimension of $\mu$};
moreover, $\dim_H K=\min \{d,\dim_S K\}$, where $\dim_S K=-T(0)$ is the {\it similarity dimension of $K$}.
 The ESC is satisfied by broad families of ovelapping IFSs.  For instance, it is satisfied by every algebraic IFS on $\R$ that does not allow exact overlaps \cite{Hochman2014}. Motivated by the achievements in \cite{Hochman2014,Hochman2019+},  we  introduce the following.

\begin{de}
\label{de-1.1}
An IFS $\Phi=\{\phi_i\}_{i=1}^\ell$ of similarities on $\R^d$ is said to be dimensional regular \footnote{ A central conjecture in fractal geometry asserts that every IFS of similarities on $\R$ is dimensional regular,  unless  it has exact overlaps.  In addition to the contributions \cite{Hochman2014,Hochman2019+} of Hochman,    Rapaport \cite{Rapaport2020} recently established the conjecture for those IFS on $\R$ with algebraic contractions and arbitrary translations.}
 if, for every  probability vector ${\bf p}=(p_1,\ldots, p_\ell)$ with strictly positive entries, the self-similar measure $\mu=\mu_{\Phi, {\bf p}}$ generated by $\Phi$ and ${\bf p}$ has dimension given by
$$
\dim_H\mu=\min\{d, \dim_S\mu\}.
$$
\end{de}

Our first result  says that if $\mu$ is a self-similar measure generated by a dimensional regular IFS, then the expression of the $L^q$-spectrum of $\mu$ over $[0, 1]$  can be explicitly determined and, moreover, the multifractal formalism holds for $\mu$ on the range of $q$ between $0$ and $1$.

\begin{thm}
\label{thm-1.1}
Let $\Phi=\{\phi_i\}_{i=1}^\ell$ be a dimensional regular IFS of similarities on $\R^d$   with ratios $r_1,\ldots, r_\ell$, and ${\bf p}=(p_1, \ldots, p_\ell)$ a probability vector with strictly positive entries. Let $\mu$ be the self-similar measure generated by $\Phi$ and ${\bf p}$,  and let $T(q)$, $q\in \R$,  be defined as in \eqref{Lqspec}. Then the following statements hold.
\begin{itemize}
\item[(1)]  The $L^q$-spectrum of $\mu$ on $[0,1]$ is given as follows:
\begin{itemize}
\item[(a)] If $T'(1)\geq d$, then $\tau_\mu(q)=d(q-1)$ for $q\in [0,1]$.
\item[(b)] If $T'(1)<d$ and $T(0)\geq -d$, then $\tau_\mu(q)=T(q)$ for $q\in [0,1]$.
\item[(c)] If $T'(1)<d$ and $T(0)<-d$, set $$\tilde{q}=\inf\{q\in (0,1): T'(q)q-T(q)\leq d\}.$$
 Then
$$
\tau_\mu(q)=
\left\{ \begin{array}{ll}
T(q) & \mbox{ if }q\in [\tilde{q}, 1],\\
\displaystyle \frac{q(d+T(\tilde{q}))}{\tilde{q}}-d & \mbox{ if }q\in [0, \tilde{q}).
\end{array}
\right.
$$
\end{itemize}
\item[(2)] $\tau_\mu$ is differentiable on $(0,1]$.  Moreover, for every $\alpha\in [\tau_\mu'(1),\tau_\mu'(0+)]$,
$$
\dim_HE_\mu(\alpha)= \tau^*_\mu(\alpha).
$$
\end{itemize}
\end{thm}

Putting Theorem \ref{thm-1.1} together with those  aforementioned  results in \cite{Shmerkin2019, Feng2007},  we see that if $d=1$ and the ESC holds, then the expression of $\tau_\mu$ is now known over $\R_+$, and the multifractal formalism holds  for $\mu$ at any $\alpha\in [\tau_\mu'(\infty), \tau_\mu'(0+)]
 $,  except a possible interval  $[\tau_\mu'(q_0+),\tau_\mu'(q_0-))$, where $q_0$ (if it exists) is the unique value in $(1,\infty)$  so that $q_0-1=T(q_0)$ and $\tau_\mu$ is not differentiable at $q_0$.  Furthermore by Remark \ref{rem-A1}(ii), the multifractal formalism holds  for $\mu$  at $\alpha=\tau'_\mu(q_0+)$. From the formula \eqref{e-Shmerkin}, one can check that  such $q_0$ exists only in the situation that  $T'(1)> 1$ and  $r_i< p_i$ for some  $i\in \{1,\ldots, \ell\}$.

A statement similar to Theorem \ref{thm-1.1} was obtained by the authors for almost all self-similar measures in  \cite[Theorem 6.4]{BarralFeng2013}; the proof of Theorem \ref{thm-1.1} turns essentially identical to that of \cite[Theorem 6.4]{BarralFeng2013}, except that at some point it requires the result of Shmerkin and Solymyak \cite[Theorem 5.1, Remark 5.2]{ShmerkinSolomyak2016} that $\tau'_\mu(1+)=\dim_H\mu$ for every self-similar measure instead of the results by Falconer \cite{Falconer1999} and Jordan, Pollicott and Simon~\cite{JordanPollicottSimon2006} on the $L^q$-spectrum and the Hausdorff dimension of almost all self-similar measures. For the sake of completeness and the reader's convenience, we include a  full proof of Theorem \ref{thm-1.1} in Appendix~\ref{S-A}.  We emphasize  that Theorem  \ref{thm-1.1} can be extended to more general measures, including the push-forwards of quasi-Bernoulli measures, the convolutions of certain self-similar measures and a class of dynamical driven self-similar measures; moreover, the multifractal formalism also hold for some of these measures on some range of $q>1$. For details, see the remarks and comments in the end of Section \ref{S-A}.

Our other results establish that the AWSC holds  for several   new classes of IFSs of similarities, either under the ESC, or for certain algebraic systems.

\begin{thm}
\label{thm-1.2}
Let $\Phi=\{\phi_i\}_{i=1}^\ell$ be an IFS of similarities on $\R$ with ratios $r_1,\ldots, r_\ell$. Suppose $\Phi$ satisfies the ESC. Then
$\Phi$ satisfies the AWSC if and only if $\sum_{i=1}^\ell r_i\leq 1$.
\end{thm}

\begin{thm}
\label{thm-1.3}
Let $\Phi=\{\phi_i\}_{i=1}^\ell$ be a homogeneous IFS on $\R$ of the form $$\phi_i(x)=\frac{x}{\beta}+a_i,\qquad i=1,\ldots, \ell,$$
where $\beta$ is an algebraic integer with $|\beta|>1$ and $a_i$ are algebraic numbers.
Then the following properties hold.
\begin{itemize}
\item[(i)]
$\Phi$ satisfies the AWSC if and only if $$\lim_{n\to \infty}\frac{\log N_n}{n\log |\beta|}\leq 1,$$ where
$
N_n:=\#\{\phi_u:\; u\in \{1,\ldots, \ell\}^n\}
$.
In particular if $\ell\leq |\beta|$, then $\Phi$ satisfies the AWSC.
\item[(ii)] Let $K$ denote the attractor of $\Phi$. Then $$\dim_HK= \min\left\{1, \lim_{n\to \infty}\frac{\log N_n}{n\log
|\beta|}\right\}.$$
\end{itemize}

\end{thm}

\begin{thm}
\label{thm-1.4}
Let $\Phi=\{\phi_i\}_{i=1}^\ell$ be an IFS on $\R^d$ of the form $$\phi_i(x)=\frac{x}{m_i}+a_i,\qquad i=1,\ldots, \ell,$$
where $m_i\in \Z$ with $|m_i|>1$ and $a_i\in {\Bbb Q}^d$.
Then $\Phi$ satisfies the AWSC.
\end{thm}

Theorems \ref{thm-1.3}-\ref{thm-1.4} provide new examples of overlapping algebraic IFS to which the results of \cite{Feng2012} can be applied to get the validity of the multifratal formalism over $[\tau_\mu'(\infty), \tau_\mu'(0^+)]$.

  The proof of Theorem \ref{thm-1.4} is elementary and short, whilst the proofs of  Theorems \ref{thm-1.2} and \ref{thm-1.3} rely on a recent deep result of Shmerkin \cite{Shmerkin2019} (see Theorem \ref{thm-S}) on the $L^q$-spectrum of self-similar measures. Since the IFS $\Phi$ considered in Theorem~\ref{thm-1.3} might  have exact overlaps in the iterations, in order to apply Shmerkin's result  we need to construct certain higher dimensional affine IFS  which is ``dual'' to $\Phi$ and satisfies some  separation property. The approach  is a bit long and delicate.   We remark that Theorem \ref{thm-1.3}(ii) can be alternatively derived from the separation property of this dual IFS and a result of Hochman \cite{Hochman2014}; see Remark \ref{rem-5.2}.

  In Theorem \ref{thm-1.3}, we have assumed $\beta$  to be an algebraic integer. It remains an interesting question whether one can only assume $\beta$ to be an algebraic number.  Meanwhile, it would be desirable to know if  Theorems \ref{thm-1.3}-\ref{thm-1.4} could be extended to IFSs of similitudes on $\R^d$ under mild assumptions.

 The paper is organized as follows. In Section \ref{S-2}, we give some preliminaries  about self-similar sets and measures, and present the definitions of various separation conditions. In Section~\ref{S-3}, we present the aforementioned  result of Shmerkin.   In Sections \ref{S-4}-\ref{S-6}, we prove Theorems \ref{thm-1.2}-\ref{thm-1.4} respectively. In Appendix \ref{S-A}, we prove Theorem \ref{thm-1.1}  and present some of its extensions.

\section{Preliminaries and  separation conditions}
\label{S-2}

In this section, we introduce the basic concepts and definitions to be used in the paper.

A mapping $\phi:\; \R^d\to \R^d$ is called a {\it similarity} if $\phi(x)=r Ux+a$, where $r>0$, $U$ is a $d\times d$ orthogonal matrix and $a\in \R^d$; and in such case $r$ is called  the {\it contraction ratio} of $\phi$.

In the remaining part of this section, let $\Phi=\{\phi_i\}_{i=1}^\ell$ be  an {\it iterated function system (IFS) of similarities on $\R^d$}, i.e.,  a finite collection of contracting similarities on $\R^d$. The {\it attractor} $K$ of $\Phi$ is the unique non-empty compact subset of $\R^d$ such that
$$
K=\bigcup_{i=1}^\ell \phi_i(K).
$$
Alternatively, $K$ is called the {\it self-similar set} generated by $\Phi$.  It is well known (\cite{Hutchinson1981}) that for every probability vector ${\bf p}=(p_1,\ldots,p_\ell )$,  there is a unique Borel probability measure $\mu$ on $\R^d$ satisfying the following relation:
$$
\mu=\sum_{i=1}^\ell p_i\mu\circ\phi_i^{-1}.
$$
The measure $\mu$ is called the {\it self-similar measure} associated with $\Phi$ and ${\bf p}$, and is supported on $K$.

The dimension theory of self-similar sets and measures have been developed under various separation conditions.
 The best known separation condition for IFS is the open set condition (\cite{Hutchinson1981}). Recall that  $\Phi$ is said to satisfy the {\it open set condition} (OSC) if there exists an non-empty open set $U\subset \R^d$ such that $\phi_i(U)$,   $i=1,\ldots,\ell$,  are disjoint subsets of $U$.

 Next we recall the definitions of two other existing separation conditions (WSC and AWSC) in the literature.  Let $r_i$ denote the contraction ratio of $\phi_i$, $i=1,\ldots, \ell$. For $u=u_1\ldots u_k\in \{1,\ldots, \ell\}^k$, write $r_u:=r_{u_1}\cdots r_{u_k}$ and $\phi_u:=\phi_{u_1}\circ \cdots\circ \phi_{u_k}$. For $x\in \R^d$ and $r>0$, let $B(x,r)$ denote the closed ball of radius $r$ with center $x$.
 Set for $n\in \N$,
\begin{equation}
\label{e-w1}
W_n:=\left\{u_1\ldots u_k\in \{1,\ldots, \ell\}^k: \; k\geq 1,\; r_{u_1}\cdots r_{u_k}\leq 2^{-n}<r_{u_1}\cdots r_{u_{k-1}}\right\}.
\end{equation}
and
\begin{equation}
t_n:=\sup_{x\in \R^d}\#\{\phi_u:\; u\in W_n,\; \phi_u(K)\cap B(x, 2^{-n})\neq \emptyset\},
\end{equation}
where $\#$ stands for  cardinality.

\begin{de}\label{AWSC}
We say that $\Phi$ satisfies the weak separation condition (WSC) if the sequence $(t_n)$ is bounded; and say that $\Phi$ satisfies the asymptotically weak separation condition
(AWSC) if
$
\lim_{n\to \infty}\frac{1}{n}\log t_n=0.
$
\end{de}

The notions of WSC and AWSC were  introduced respectively in \cite{LauNgai1999} and \cite{Feng2007}.  Finally, we present the definition of the exponential separation condition (ESC) introduced in \cite{Hochman2014, Hochman2019+}. Define the distance between similarities $\psi=rU+a$ ad $\psi'=r'U'+a'$ by
$$
\rho(\psi, \psi')=|\log r-\log r'|+\|U-U'\|+\|a-a'\|,
$$
where $\|\cdot\|$ denotes the Euclidean or operator norm as appropriate.  For the IFS $\Phi=\{\phi_i\}_{i=1}^\ell$, set for $n\in \N$,
$$
\Delta_n=\min\{\rho(\phi_u, \phi_v):\; u, v\in \{1,\ldots, \ell\}^n,\; u\neq v\}.
$$

\begin{de}
\label{ESC}
Say that $\Phi$ satisfies the exponential separation condition (ESC) if there exists $c\in (0,1)$ such that $\Delta_n\geq c^n$  for infinitely many $n$.
\end{de}

For completion, we introduce the following.
\begin{de}
\label{WESC}
Say that $\Phi$ satisfies the weak exponential separation condition (WESC) if there exists $c\in (0,1)$ such that  $\widetilde{\Delta}_n\geq c^n$  for infinitely many $n$, where $$
\widetilde{\Delta}_n=\min\{\rho(\phi_u, \phi_v):\; u, v\in \{1,\ldots, \ell\}^n,\; \phi_u\neq \phi_v\}.
$$
\end{de}

 It is worth pointing out that the WESC is satisfied by all IFSs on $\R$ defined by algebraic parameters (see \cite[Theorem 1.5]{Hochman2014}).

  One has the implications $\mbox{OSC}\Longrightarrow\mbox{WSC}\Longrightarrow\mbox{AWSC}\Longrightarrow\mbox{WESC}$ and $\mbox{OSC}\Longrightarrow\mbox{ESC}\Longrightarrow\mbox{WESC}$.  The  implication $\mbox{OSC}\Longrightarrow\mbox{WSC}$ was proved in \cite{Nguyen2002}, while the other ones are easily checked from the definitions. We remark that
the conditions of  WSC, AWSC and WESC allow {\it exact overlaps} in the iterations whilst the ESC does not allow that. Recall that $\Phi$ is said to have an exact overlap in the iterations if $\Delta_n=0$ for some $n\in \N$, or equivalently, $\phi_u=\phi_v$ for some $u, v\in \bigcup_{n=1}^\infty\{1,\ldots \ell\}^n$ with $u\neq v$.

\section{Shmerkin's result on the $L^q$ spectrum of self-similar measures}
\label{S-3}
Here we present the  following result of Shmerkin on  the $L^q$-spectrum of self-similar measures on the line.
\begin{thm}[\cite{Shmerkin2019}]
\label{thm-S}
Let $\Phi=\{\phi_i\}_{i=1}^\ell$ be an IFS on $\R$ with ratios $r_1,\ldots, r_\ell$. Let $\mu$ be the self-similar measure generated by $\Phi$ and a given probability vector $(p_1,\ldots, p_\ell)$.
Then \begin{itemize}
\item[(i)] If $\Phi$ satisfies the ESC, then for each $q\geq 1$, $\tau_\mu(q)=\min\{q-1, T(q)\}$, where $T$ is defined by
$$
\sum_{i=1}^\ell p_i^q |r_i|^{-T(q)}=1.
$$
\item[(ii)] If $\Phi$ is  algebraic in the sense that  the ratios and the translation parts of $\phi_i$  are all algebraic numbers, and moreover suppose that $r_1=\cdots=r_\ell=r$,  then for each $q\geq 1$,
$\tau_\mu(q)=\min\{q-1, \lim_{n\to \infty}T_n(q)\}$, where
$T_n$ is defined by
$$
\sum_{u\in \Sigma_n/\sim}\tilde{p}_u^q |r|^{-T_n(q)}=1.
$$
Here $\sim$ is the equivalence relation on $\Sigma_n$ defined by $u\sim v$ if $\phi_u=\phi_v$, and for $u\in \Sigma_n$, $\tilde{p}_u:=\sum_{v\in \Sigma_n: \;\phi_v=\phi_u} p_v$.
\end{itemize}
\end{thm}

 Part (i) of the above theorem is just Theorem 6.6 in \cite{Shmerkin2019},  whilst part (ii) is a direct consequence of \cite[Proposition 6.8]{Shmerkin2019}.   Note that in \cite{Shmerkin2019} Shmerkin obtained a more general statement for a class of measures called {\it dynamically driven self-similar measures} which includes, for instance, the convolutions  of  $\times 2$- and $\times 3$-invariant  self-similar measures on the 1-dimensional torus. This yields an affirmative answer to  Furstenberg's conjecture on the Hausdorff dimension of the intersections of  $\times 2$- and $\times 3$-invariant sets, as well as  a proof of a strong version of  Furstenberg's conjecture on the Hausdorff dimension of the algebraic sums of such sets (see \cite[Section~7]{Shmerkin2019}).

\section{The proof of Theorem \ref{thm-1.2}}
We will need the following lemma.
\begin{lem}
\label{lem-k1}
Let $\mu$ be a compactly supported measure on $\R^d$. Then for any $q>0$,
$$
\tau_\mu(q)\leq q\liminf_{r\to 0}\frac{\log \sup_{x\in \R^d} \mu(B(x, r))}{\log r}.
$$
\begin{proof}
It simply follows from the definition of $\tau_\mu(q)$ and the fact  that  $$\Theta_\mu(q, r)\geq \sup_{x\in \R^d} \mu(B(x, r))^q,$$
where $\Theta_\mu(q, r)$ is defined as in \eqref{e-tz1}.
\end{proof}
\end{lem}

\begin{proof}[Proof of Theorem \ref{thm-1.2}]
Let $s$ be the unique number so that $\sum_{i=1}^\ell r_i^s=1$.  Let $m$ be the infinite Bernoulli product measure on $\Sigma=\{1,\ldots, \ell\}^\N$ associated with the probability weight $(r_1^s, \ldots, r_\ell^s)$. Let $W_n$, $n\in \N$, be defined as in \eqref{e-w1}. Notice that for each $n$,  $W_n$ is a section of $\Sigma$ in the sense that
$\{[u]:\; u\in W_n\}$ is a partition of $\Sigma$, where $[u]$ stands for the cylinder set associated to $u$, i.e.
$$[u]=\left\{(x_n)_{n=1}^\infty\in \Sigma:\; x_i=u_i \mbox{ for } 1\leq i\leq k\right\}\; \mbox{ for }u=u_1\ldots u_k.
$$
Hence
\begin{equation}
\label{e-b1}
\sum_{u\in W_n}m([u])=1.
\end{equation}
 Furthermore by the definition of $W_n$,  for each $u\in W_n$ we have
\begin{equation}
\label{e-b2}
2^{-ns}\min_{1\leq i\leq \ell} r_i^s<m([u])\leq 2^{-ns}.
\end{equation}
Combining \eqref{e-b1} with \eqref{e-b2} yields that
\begin{equation}
\label{e-k10}
\#W_n\geq 2^{ns},\quad n\in \N.
\end{equation}

Since $\Phi$ satisfies the ESC, there are no exact overlaps (i.e., $\phi_u\neq \phi_v$ for any distinct words $u$ and $v$) in the iterations of $\Phi$.   Set
\begin{equation}
\label{e-b3}
t_n:=\sup_{x\in \R^d}\#\{u\in W_n: \phi_u(K)\cap [x-2^{-n}, x+2^{-n}]\neq \emptyset\}.
\end{equation}
To prove Theorem \ref{thm-1.2},  it is equivalent to show that $\lim_{n\to \infty}\frac{1}{n}\log t_n=0$ if and only if $r_1+\cdots+r_\ell\leq 1$.

First assume that $r_1+\cdots+r_\ell>1$. In this case  $s>1$. Cover  $K$  by $2^n \mbox{diam}(K)+1$  intervals of length $2^{-n}$.  For each $u\in W_n$, $\phi_u(K)$ intersects at least one of these intervals. Conversely by \eqref{e-b3}, every such interval intersects $\phi_u(K)$ for at most $t_n$ elements $u$ in $W_n$. Hence
$$\#W_n\leq t_n (2^n \mbox{diam}(K)+1).$$
So by \eqref{e-k10},
$$t_n\geq  \frac{\#W_n}{2^n \mbox{diam}(K)+1}\geq  \frac{2^{ns}}{2^n \mbox{diam}(K)+1},
$$
 leading to the inequality $$\liminf_{n\to \infty}\frac{1}{n}\log t_n\geq  (s-1) \log 2>0,$$ and so $\Phi$ does not satisfy the AWSC.

Next assume that $r_1+\cdots+r_\ell\leq 1$. In this case, $s\leq 1$. Let  $\mu$ denote the self-similar measure generated by $\Phi$ and the probability vector $(r_1^{s}, \ldots, r_\ell^{s})$.
Take $R>\mbox{diam}(K)+1$.  Let $y\in \R$.  Clearly for each $u\in W_n$ with $\phi_u(K)\cap [y-2^{-n},  y+2^{-n}]\neq \emptyset$, we have $$\mbox{diam}(\phi_u(K))\leq 2^{-n}\mbox{diam}(K)\leq 2^{-n} (R-1)$$ and so
\begin{equation}\label{e-b5.0}
\phi_u(K)\subset [y -2^{-n}R,  y+ 2^{-n}R].
\end{equation}
It follows that
\begin{equation}
\label{e-b4}
\begin{split}
&\sup_{x\in \R}  \#\{u\in W_n:\; \phi_u(K)\subset [x-2^{-n}R, x+2^{-n}R]\}\\
&\mbox{}\quad  \geq \sup_{x\in \R}  \#\{u\in W_n:\; \phi_u(K)\cap  [x-2^{-n}, x+2^{-n}]\neq \emptyset\} \\
&\mbox{}\quad  =t_n.
\end{split}
\end{equation}
Since $W_n$ is a section of $\Sigma$, it follows that $$\mu=\sum_{u\in W_n} r_u^s \mu\circ \phi_u^{-1}$$
 (see e.g. \cite[Lemma~2.2.4]{BishopPeres17} for a proof),  hence
\begin{eqnarray*}
&&\sup_{x\in \R}\mu([x -2^{-n}R, x+ 2^{-n}R])\\
& &\mbox{}\quad=  \sup_{x\in \R} \sum_{u\in W_n}r_u^s \mu \left(\phi_u^{-1}[x -2^{-n}R, x+ 2^{-n}R]\right)\\
&&\mbox{}\quad\geq C2^{-ns}  \sup_{x\in \R}  \sum_{u\in W_n} \mu\left(\phi_u^{-1}[x -2^{-n}R, x+ 2^{-n}R]\right)\\
&&\mbox{} \quad\geq C 2^{-ns} \sup_{x\in \R}  \#\{u\in W_n:\;  K\subset \phi_u^{-1}[x-2^{-n}R, x+2^{-n}R]\}\\
 &&\mbox{}\quad\geq C 2^{-ns} t_n\qquad \mbox{(by \eqref{e-b4})},
\end{eqnarray*}
where $C:=\min_{1\leq i\leq \ell}r_i^s$. Applying Lemma \ref{lem-k1}  yields that
\begin{equation}
\label{e-b5}
\tau_\mu(q)\leq q \liminf_{n\to \infty} \frac{\log(C 2^{-ns} t_n) }{\log(2^{-n}R)}\leq q (s-\gamma),\quad q>0,
\end{equation}
where $\gamma=\displaystyle\limsup_{n\to \infty}\frac{\log t_n}{n\log 2}$.

Suppose on the contrary that $\Phi$ does not satisfy the AWSC.  Then $\gamma>0$. Since $\sum_{i=1}^\ell  (r_i^{s})^qr_i^{-s(q-1)}=1$, by Theorem \ref{thm-S}(i),  $$\tau_\mu(q)=\min\{q-1, s(q-1)\}=s(q-1)$$ for each $q>1$. This contradicts \eqref{e-b5} since $q(s-\gamma)<s(q-1)$ for sufficiently large~$q$.
\end{proof}

\section{The proof of Theorem \ref{thm-1.3}}
\label{S-4}

Throughout this section, let $\Phi=\{\phi_i\}_{i=1}^\ell$ be a homogeneous IFS on $\R$ of the form $$\phi_i(x)=\frac{x}{\beta}+a_i,\qquad i=1,\ldots, \ell,$$
where $\beta$ is an algebraic integer with $|\beta|>1$ and $a_i$ are algebraic numbers. Let $K$ denote the attractor of $\Phi$.
Set $\Sigma_n=\{1,\ldots, \ell\}^n$ for $n\in \N$ and $\Sigma=\{1,\ldots, \ell\}^\N$. We first state a key result that will be needed in the proof of Theorem \ref{thm-1.3}.

\begin{pro}
\label{pro-a1}
There exist an integer $k\geq 1$ and a homogeneous affine IFS $\Psi=\{\psi_i(x)=Ax+b_i\}_{i=1}^\ell$ on $\R^k$ so that the following properties hold:
\begin{itemize}
\item[(i)] For any finite words $I, J$ over $\{1,\ldots, \ell\}$, $\psi_I=\psi_J$ if and only if $\phi_I=\phi_J$.
\item[(ii)] Let $\widetilde{K}$ denote the attractor of $\Psi$. Set for $n\in \N$,
\begin{equation}\label{e-k1}
\kappa_n:=\sup_{z\in \R^k}\#\left\{\psi_J:\; J\in \Sigma_n,\; \psi_J(\widetilde{K})\cap A^n(B(z, 1))\neq \emptyset\right\},
\end{equation}
where   $B(z, 1)$ stands for the closed unit ball centered at $z$.
Then $$\lim_{n\to \infty}\frac{1}{n}\log \kappa_n=0.$$
 Moreover, the sequence $(\kappa_n)$ is  bounded if no conjugates of $\beta$ lie on the unit circle $\{|z|=1\}$.
\end{itemize}
\end{pro}
The above result says that one can always find an IFS $\Psi$ ``dual'' to $\Phi$, so that $\Psi$ satisfies certain asymptotically weak separation property.  The proof of this result will be postponed until the end of this section. Below we first apply it to prove Theorem  \ref{thm-1.3}.

\begin{proof}[Proof of Theorem \ref{thm-1.3}]
Recall that $N_n=\#\{\phi_u:\; u\in \Sigma_n\}$ for $n\in \N$. It is easy to check that the sequence $(N_n)$ is sub-multiplicative, i.e. $N_{n+m}\leq N_nN_m$ for all $n,m\in \N$. Hence the following limit exists:  $$
s:=\lim_{n\to \infty}\frac{\log N_n}{n\log |\beta|}.$$
  Set for $n\in \N$,
\begin{equation}\label{e-k2}
t_n:=\sup_{x\in \R}\#\{\phi_u:\; u\in \Sigma_n,\; \phi_u(K)\cap[x-|\beta|^{-n}, x+|\beta|^{-n}]\neq \emptyset\}.
\end{equation}
 Write
$$
\gamma:=\limsup_{n\to \infty}\frac{1}{n}\log t_n.
$$
First notice that $t_{n+1}\geq t_n$ for each $n$. To see this, for given $n\in \N$ let $x\in \R$ be a point attaining the supremum in \eqref{e-k2} and let $S_1,\ldots, S_{t_n}$ be $t_n$ different elements in $\{\phi_u:\; u\in \Sigma_n\}$ so that $S_i(K) \cap[x-|\beta|^{-n}, x+|\beta|^{-n}]\neq \emptyset$. Then $\phi_1\circ S_1,\ldots, \phi_1\circ S_{t_n}$ are different elements in $\{\phi_u:\; u\in \Sigma_{n+1}\}$ so that $$(\phi_1\circ S_i)(K)\cap [\phi_1(x)-|\beta|^{-n-1}, \phi_1(x)+|\beta|^{-n-1}]\neq \emptyset,\quad  i=1,\ldots, t_n,$$ hence by definition,  $t_{n+1}\geq t_n$.
  As a consequence,  for each $p\in \N$,
\begin{equation}
\label{e-k3}
\limsup_{n\to \infty}\frac{1}{np}\log t_{np}=\limsup_{n\to \infty}\frac{1}{n}\log t_n=\gamma.
\end{equation}
Below we divide the remaining part of our proof  into 3 steps.

{\sl Step 1. $\Phi$ does not satisfy the AWSC if $s>1$.}

To see this, assume that $s>1$. For $n\in \N$, cover $K$ by $|\beta|^n \mbox{diam}(K)+1$ intervals with length $|\beta|^{-n}$.  By the pigeon-hole principle, one of these intervals intersects $\phi_u(K)$ for at least $\frac{N_n}{|\beta|^n {\rm diam}(K)+1}$ many different maps $\phi_u$ in $\{\phi_u:\; u\in \Sigma_n\}$. Therefore we have
$$t_n\geq \frac{N_n}{|\beta|^n \mbox{diam}(K)+1},$$
which implies that $$
\gamma=\limsup_{n\to \infty}\frac{\log t_n}{n}\geq \limsup_{n\to \infty}\frac{\log (N_n/|\beta|^n)}{n}\geq (s-1)\log |\beta|>0,$$
 and so $\Phi$ does not satisfy the AWSC.

{\sl Step 2. $\Phi$  satisfies the AWSC if $s\leq 1$.}

Let $\Psi=\{\psi_i(x)=Ax+b_i\}_{i=1}^\ell$ be the auxiliary affine IFS constructed in Proposition \ref{pro-a1} and let $(\kappa_n)$ be defined as in \eqref{e-k1}.

Fix a large integer $L$. Define an equivalence relation $\sim$ on $\Sigma_L$ by $u\sim v$ if $\phi_u=\phi_v$. Let $\underline{u}$ denote the equivalence class containing $u$. Denote $\phi_{\underline{u}}=\phi_u$ and $\psi_{\underline{u}}=\psi_u$. Set $$\mathcal J=\Sigma_L/\sim:=\{\underline{u}:\; u\in \Sigma_L\}.$$
Then $\# \mathcal J =N_L$.
Build two IFSs $\Phi_L$ and $\Psi_L$ by
$$
\Phi_L=\{\phi_{\underline{u}}:\; \underline{u}\in \mathcal J\},\quad \Psi_L=\{\psi_{\underline{u}}:\; \underline{u}\in \mathcal J\}.
$$
Clearly $\Phi_L$ has the attractor $K$, and $\Psi_L$ has the attractor $\widetilde{K}$. Let $\mu$ be the self-similar measure generated by $\Phi_L$ and the uniform probability vector
${\bf p}=(\frac{1}{N_L},\ldots, \frac{1}{N_L})$. Meanwhile let $\eta$ be the self-affine measure generated by $\Psi_L$ and ${\bf p}$.

For each $n\in \N$, define an equivalence relation $\sim$ on $\mathcal J^n$ by $u_1\ldots u_n\sim  v_1\ldots v_n$ if $\phi_{ u_1\ldots u_n}=\phi_{v_1\ldots v_n}$. For $w\in \mathcal J^n$, define
$$
\tilde{p}_w=\frac{1}{(N_L)^n}\#\{w'\in \mathcal J^n:\; \phi_{w'}=\phi_w\}.
$$
By similarity we have
$$
\mu=\frac{1}{(N_L)^n}\sum_{w\in \mathcal J^n} \mu\circ \phi_{w}^{-1}=\sum_{\underline{w}\in \mathcal J^n/\sim} \tilde{p}_w\mu\circ \phi_{w}^{-1},
$$
and so by property (i) in Proposition \ref{pro-a1},
\begin{equation}\label{e-k7'}
\eta=\frac{1}{(N_L)^n}\sum_{w\in \mathcal J^n} \eta\circ \psi_{w}^{-1}=\sum_{\underline{w}\in \mathcal J^n/\sim} \tilde{p}_w\eta\circ \psi_{w}^{-1}.
\end{equation}
We claim that there exists $M>0$ such that for all $n\in \N$ and $w\in \mathcal J^n$,
\begin{equation}
\label{e-k6}
\left(\frac{1}{N_L}\right)^n\leq  \tilde{p}_w\leq M\left(\frac{\kappa_L}{N_L}\right)^n.
\end{equation}
The first inequality in \eqref{e-k6} is trivial, which follows directly from the definition of $\tilde{p}_w$. To prove the other inequality we need to use the asymptotically weak separation like property of  $\Psi_L$.  Notice that
\begin{eqnarray*}
&\mbox{}&\sup_{z\in \R^k} \#\left\{u\in \mathcal J:\; \psi_u(\widetilde{K})\cap A^L (B(z,1))\neq \emptyset \right\}\\
&\mbox{}&\quad =\sup_{z\in \R^k} \#\left\{\psi_u:\; u\in \Sigma_L,\; \psi_u(\widetilde{K})\cap A^L (B(z,1))  \neq \emptyset\right\}\\
&\mbox{}&\quad = \kappa_L.
\end{eqnarray*}
Since $A$ is contracting, it follows that for each  $n\in \N$ and $z\in \R^k$, $$A^{nL} (B(z,1))\subset  A^L \left(B\left(A^{(n-1)L}z,1\right)\right)$$ and so,
\begin{eqnarray*}
&& \#\left\{u\in \mathcal J:\; \psi_u(\widetilde{K})\cap A^{nL} (B(z,1))\neq \emptyset \right\}\\
&&\mbox{}\quad \leq  \#\left\{\psi_u:\; u\in \Sigma_L,\; \psi_u(\widetilde{K})\cap A^L \left(B\left(A^{(n-1)L}z,1\right)\right)  \neq \emptyset\right\}\\
&&\mbox{}\quad  \leq  \kappa_L.
\end{eqnarray*}
Hence by the self-affinity of $\eta$, for each $n\in \N$ and $z\in \R^k$,
\begin{eqnarray*}
\label{e-k7}
\eta\left(A^{nL}(B(z,1)\right)&=& \frac{1}{N_L} \sum_{u\in \mathcal J:\; \widetilde{K}\cap  \psi_u^{-1} (A^{nL}(B(z,1)))\neq \emptyset }\eta\left(\psi_u^{-1}(A^{nL}(B(z,1)))\right)\\
&=&\frac{1}{N_L} \sum_{u\in \mathcal J:\; \psi_u(\widetilde{K})\cap A^{nL}(B(z,1))\neq \emptyset }\eta\left(\psi_u^{-1}(A^{nL}(B(z,1)))\right)\\
& \leq & \frac{\kappa_L}{N_L} \sup_{x\in \R^k} \eta\left(A^{(n-1)L}(B(x,1))\right),
\end{eqnarray*}
where in the last inequality we used the fact that for each $u\in {\mathcal J}$, the mapping $\psi_u^{-1}$ is of the form $A^{-L}x+c_u$, and so $\psi_u^{-1}(A^{nL}(B(z,1)))=A^{(n-1)L}(B(z',1))$ with $z'=z+A^{-(n-1)L}c_u$.
Iterating the above inequality yields
 \begin{equation}
 \label{e-star}
 \eta\left(A^{nL}(B(z,1))\right)\leq \left(\frac{\kappa_L}{N_L}\right)^n\sup_{x\in \R^k} \eta(B(x,1))\leq \left(\frac{\kappa_L}{N_L}\right)^n.
 \end{equation}
Let $M$ be the smallest number of closed unit balls that are needed to cover $\widetilde{K}$. Let $w\in \mathcal J^n$.  Since $\psi_w(\widetilde{K})$ can be covered by $M$ ellipsoids of the form $A^{nL}(B(z,1))$, from \eqref{e-star} we obtain that
$$
\eta(\psi_w(\widetilde{K}))\leq M\left(\frac{\kappa_L}{N_L}\right)^n.
$$
However  by the self-affinity property \eqref{e-k7'}, $$\eta(\psi_w(\widetilde{K}))\geq \tilde{p}_w \eta(\psi_w^{-1} (\psi_w(\widetilde{K})) =\tilde{p}_w,$$
which implies that   $\displaystyle \tilde{p}_w\leq M\left(\frac{\kappa_L}{N_L}\right)^n$. This completes the proof of \eqref{e-k6}.

Now we use \eqref{e-k6} to derive a lower bound for the $L^q$-spectrum of $\mu$. Applying Theorem \ref{thm-S}(ii) to the IFS $\Phi_L$ and $\mu$, we see that for any $q> 1$,
\begin{equation}
\label{e-k8}
\tau_\mu(q)=\min\{q-1, \lim_{n\to \infty} T_n(q)\},
\end{equation}
where $T_n(q)$ satisfies the equation $\sum_{\underline{w}\in \mathcal J^n/\sim}(\tilde{p}_w)^q |\beta|^{nL T_n(q)}=1$, i.e.
$$
T_n(q)=\frac{\log(\sum_{\underline{w}\in \mathcal J^n/\sim}(\tilde{p}_w)^q)}{-nL\log |\beta|}.
$$
Since the cardinality of $\mathcal J^n/\sim$ does not exceed $(N_L)^n$, applying \eqref{e-k6} to the above equality yields
\begin{eqnarray*}
T_n(q)&\geq & \frac{\log\left((N_L)^nM^q\left(\frac{\kappa_L}{N_L}\right)^{nq}\right)}{-nL\log |\beta|}\\
&=&q\left(\frac{\log N_L-\log \kappa_L}{L\log |\beta|}\right)-\frac{\log N_L}{L\log |\beta|}-\frac{q\log M}{nL\log |\beta|}.
 \end{eqnarray*}
 Letting $n\to \infty$ gives
 \begin{equation}
 \label{e-k9}
 \lim_{n\to \infty} T_n(q)\geq q\left(\frac{\log N_L-\log \kappa_L}{L\log |\beta|}\right)-\frac{\log N_L}{L\log |\beta|}\quad  \mbox{ for all } q>1.
 \end{equation}
Hence by \eqref{e-k8}, for each $q>1$,
\begin{equation}
\label{e-k8'}
\tau_\mu(q)\geq \min\left\{q-1, \; q\left(\frac{\log N_L-\log \kappa_L}{L\log |\beta|}\right)-\frac{\log N_L}{L\log |\beta|}\right\}.
\end{equation}

 Next we provide an upper bound for $\tau_\mu(q)$.  Recalling the definition of $(t_n)$ (see \eqref{e-k2}), we have
 \begin{eqnarray*}
 t_{nL}&=&\sup_{x\in \R}\#\{\phi_u:\; u\in \Sigma_{nL},\; \phi_u(K)\cap[x-|\beta|^{-nL}, x+|\beta|^{-nL}]\neq \emptyset\}\\
 &=& \sup_{x\in \R}\#\{\phi_u:\; u\in \mathcal J^n,\; \phi_u(K)\cap[x-|\beta|^{-nL}, x+|\beta|^{-nL}]\neq \emptyset\}.
 \end{eqnarray*}
 Pick a number $R>\mbox{diam}(K)+1$. Then for $u\in  \mathcal J^n$, $$\mbox{diam}(\phi_u(K))=|\beta|^{-nL}\mbox{diam}(K)\leq |\beta|^{-nL}(R-1),$$ hence if $\phi_u(K)\cap[x-|\beta|^{-nL}, x+|\beta|^{-nL}]\neq \emptyset$ for some $x\in \R^k$,  then
 $$
 \phi_u(K)\subset [x-|\beta|^{-nL}R, x+|\beta|^{-nL}R].
 $$
 The above observation shows that for each $x\in \R^k$,
  \begin{eqnarray*}
 &&\mu([x-|\beta|^{-nL}R, x+|\beta|^{-nL}R])\\
 &&\quad =\frac{1}{(N_L)^n}\sum_{w\in \mathcal J^n} \mu\circ \phi_{w}^{-1}([x-|\beta|^{-nL}R, x+|\beta|^{-nL}R])\\
 &&\quad \geq\frac{1}{(N_L)^n}\#\{w\in \mathcal J^n:\; \phi_w(K)\subset [x-|\beta|^{-nL}R, x+|\beta|^{-nL}R]\}\\
 &&\quad \geq\frac{1}{(N_L)^n}\#\{w\in \mathcal J^n:\; \phi_w(K)\cap [x-|\beta|^{-nL}, x+|\beta|^{-nL}]\neq \emptyset\}\\
 &&\quad  \geq\frac{1}{(N_L)^n}\#\{\phi_w:\; w\in  \Sigma_{nL},\; \phi_w(K)\cap [x-|\beta|^{-nL}, x+|\beta|^{-nL}]\neq \emptyset\}.
\end{eqnarray*}
 This implies that
 $$
 \sup_{x\in \R} \mu([x-|\beta|^{-nL}R, x+|\beta|^{-nL}R])\geq t_{nL}{(N_L)^{-n}}.
 $$
 By Lemma \ref{lem-k1}, for each $q>0$,
\begin{equation}\label{e-k4''}
\begin{split}
 \tau_\mu(q)& \leq  q\cdot \liminf_{n\to \infty} \frac{n\log N_L-\log t_{nL}}{nL\log |\beta|-\log R} \\
  & = q\left(\frac{\log N_L}{L\log |\beta|}-\frac{\gamma}{\log |\beta|} \right)\qquad (\mbox{by \eqref{e-k3}}).
 \end{split}
 \end{equation}

 Now we are ready to prove the claim that $\Phi$ satisfies the AWSC if $s\leq 1$. Suppose on the contrary that $s\leq 1$ but $\Phi$ does not satisfy the AWSC. Then $\gamma>0$.
 Since $\lim_{n\to \infty}\frac1n\log \kappa_n=0$ and $s=\displaystyle \lim_{n\to \infty}\frac{\log N_n}{n\log |\beta|}\leq 1$, we may assume that $L$ is large enough so that
\begin{equation}
\label{e-k4}
\kappa_L<e^{\gamma L/2}
\end{equation}
and
\begin{equation}
\label{e-k4'}
\frac{\log N_L}{L\log |\beta|}-\frac{\gamma}{\log |\beta|}<1.
\end{equation}
 Combining  \eqref{e-k4'} and \eqref{e-k4''} yields that
 \begin{equation}\label{e-kk}
 \tau_\mu(q)\leq q\left(\frac{\log N_L}{L\log |\beta|}-\frac{\gamma}{\log |\beta|} \right)<q-1 \quad \mbox{for sufficiently large $q$}.
 \end{equation}
By \eqref{e-kk}, \eqref{e-k8'} and \eqref{e-k4}, for sufficiently large $q$,
$$
\tau_\mu(q)\geq  q\left(\frac{\log N_L-\log \kappa_L}{L\log |\beta|}\right)-\frac{\log N_L}{L\log |\beta|}\geq q\left(\frac{\log N_L}{L\log |\beta|}-\frac{\gamma}{2 \log |\beta|} \right),
$$
which contradicts \eqref{e-kk}.  This completes the proof of Step 2.

  {\sl Step 3. $\dim_HK=\min\{1, s\}$.}

  Clearly  $\dim_HK\leq \min\{1, s\}$. We only need to prove the converse inequality. For this purpose, let us estimate $\dim_H\mu$ for the self-similar measure $\mu$ constructed in Step 2. It is well known  that for each $q>1$, $\dim_H\mu\geq \frac{\tau_\mu(q)}{q-1}$ (see e.g.  \cite[Theorem 1.4]{FanLauRao2002}). Hence by \eqref{e-k8'},
  $$
  \dim_H\mu\geq \lim_{q\to \infty}\frac{\tau_\mu(q)}{q-1}\geq \min\left\{1, \;\frac{\log N_L-\log \kappa_L}{L\log |\beta|}\right\}.
  $$
  Since $\mu$ is supported on $K$, it follows that  $\dim_HK\geq \dim_H\mu \geq \min\left\{1, \;\frac{\log N_L-\log \kappa_L}{L\log |\beta|}\right\}$. Letting $L\to \infty$ yields $\dim_HK\geq \min\{1, s\}$ and we are done.
\end{proof}

\begin{rem}
\label{rem-5.2}
Alternatively, in Step 3 of the above proof we can prove the inequality
\begin{equation}
\label{e-ine}
\dim_H\mu\geq  \min\left\{1, \;\frac{\log N_L-\log \kappa_L}{L\log |\beta|}\right\}
\end{equation}
by using \eqref{e-k6} and applying a result of Hochman. To see this, define
$$
h_G(\mu):=\lim_{n\to \infty}\frac{1}{n}\sum_{\underline{w}\in {\mathcal J}^n/\sim}(-\tilde{p}_w\log \tilde{p}_w),
$$
which is called the Garsia entropy of $\mu$. Since $\Phi_L$ is an algebraic IFS, it was implicitly proved in \cite{Hochman2014} that
\begin{equation}
\label{e-hoc}
\dim_H\mu=\min\left\{1, \frac{h_G(\mu)}{L\log |\beta|}\right\}.
\end{equation}
(See Sect.~3.4 of \cite{BreuillardVarju2015} for more explanations.) Now applying the second inequality in \eqref{e-k6} to the definition of $h_G(\mu)$  yields immediately that $h_G(\mu)\geq \log N_L-\log \kappa_L$ and so  \eqref{e-ine} follows from \eqref{e-hoc}.
\end{rem}

In the remaining part of this section, we prove Proposition \ref{pro-a1}. As usual, we use ${\Bbb Q},{\Bbb Z},{\Bbb C}$ to denote the sets of rational numbers, integers and complex numbers, respectively.  In the following lemma, we collect some elementary properties of algebraic numbers. For  a proof, see e.g. \cite{Pollard50}.
\begin{lem}
\label{lem-algebraic}
\begin{itemize}
\item[(i)] The totality of algebraic numbers forms a field, whilst the totality of algebraic integers forms a ring.
\item[(ii)] If $\theta$ is an algebraic number, there is an integer $k\neq 0$ such that $k\theta$ is an algebraic integer.
\item[(iii)] For a set of  algebraic numbers $\alpha_1, \ldots, \alpha_n$, there exists an algebraic integer $\theta$ so that $\alpha_i\in {\Bbb Q}(\theta)$ for  $1\leq i\leq n$, where  ${\Bbb Q}(\theta)$ is the field generated by ${\Bbb Q}$ and $\theta$.
\item[(iv)] Let $\alpha_1,\ldots, \alpha_n$ be all the algebraic conjugates of an algebraic number $\alpha$ with $\alpha_1=\alpha$. Then for each polynomial $f$ with rational coefficients, $f(\alpha_i)$ are algebraic conjugates of $f(\alpha)$ and moreover, $\prod_{i=1}^n f(\alpha_i)\in {\Bbb Q}$.  If additionally $\alpha$ is an algebraic integer and $f$ has integers for  coefficients, then $\prod_{i=1}^n f(\alpha_i)\in {\Bbb Z}$.
\end{itemize}
\end{lem}

Now we are ready to prove Proposition \ref{pro-a1}.

\begin{proof}[Proof of Proposition \ref{pro-a1}]
Since $\beta, a_1,\ldots, a_\ell$ are algebraic numbers, by Lemma \ref{lem-algebraic}(iii),  there exists an algebraic integer $\lambda$ so that $$\beta, a_1,\ldots, a_\ell\in \Bbb Q(\lambda),$$
where $\Bbb Q(\lambda)$ is the field generated by $\Bbb Q$ and $\lambda$. Hence there exist polynomials  $f$, $g_1,\ldots, g_\ell$ with rational coefficients so that
$\beta=f(\lambda)$ and $a_i=g_i(\lambda)$ for $i=1,\ldots, \ell$.

Let $d$ be the degree of $\lambda$, and let $\lambda_1=\lambda$,\ldots, $\lambda_d$ be the algebraic conjugates of $\lambda$.
Set for $j=1,\ldots, d$,
$$
\beta_j=f(\lambda_j),\quad a_{i,j}=g_i(\lambda_j),\quad i=1,\ldots, \ell.
$$
Then for each $j$, $\beta_j$ is an algebraic conjugate of $\beta$ (so it is an algebraic integer) and $a_{i,j}$ is an algebraic conjugate of $a_i$ for each $i$.

Taking a permutation of the indices $\{2,\ldots, d\}$ if necessary, we may assume that there exist two integers $m, m'$ with  $1\leq m\leq m'\leq d$ such that
$|\beta_j|>1$ for $1\leq j\leq m$, $|\beta_j|=1$ for $m< j\leq m'$ and $|\beta_j|<1$ for $m'<j\leq d$.

For $1\leq i\leq \ell$ and $1\leq j\leq d$, define $h_{i,j}:\Bbb C\to \Bbb C$ by
$$
h_{i,j}(z)=\frac{z}{\beta_j}+a_{i,j}.
$$
For $u=u_1\ldots u_n\in \Sigma_n$, write $h_{u, j}=h_{u_1,j}\circ \cdots \circ h_{u_n, j}$. Notice that $h_{u,1}=\phi_u$ (here we view $\phi_u$ as a mapping on $\Bbb C$). We will use the following simple but important algebraic property:
\begin{itemize}
\item[({\bf P})] If $h_{u,j}=h_{v,j}$ for some $u,v\in \Sigma_n$ and some $j\in \{1,\ldots, d\}$, then $h_{u,j'}=h_{v,j'}$ for all $j'\in \{1,\ldots, d\}$.
\end{itemize}

This property is a consequence of Lemma \ref{lem-algebraic}(iv). Indeed, for $u, v\in \Sigma_n$ and $j\in \{1,\ldots, d\}$,
$$
h_{u, j}(z)-h_{v, j}(z)=\beta_j^{-n} \sum_{p=1}^n(a_{u_p, j}-a_{v_p, j})\beta_j^{n-p+1}=\beta_j^{-n}H(\lambda_j),
$$
where $H$ is a polynomials with rational coefficients given by
\begin{equation}
\label{e-h1}
H(x)=\sum_{p=1}^n (g_{u_p}(x)-g_{v_p}(x)) f(x)^{n-p+1}.
\end{equation}
By Lemma \ref{lem-algebraic}(iv), if $H(\lambda_j)=0$ for some $j\in \{1,\ldots, d\}$, then $H(\lambda_{j'})=0$ for all $j'\in \{1,\ldots, d\}$. This proves the property ({\bf P}).

Now define an  IFS $\Psi=\{\psi_i\}_{i=1}^\ell$ on $\Bbb C^m$ by
$$
\psi_i(z_1,\ldots, z_m)=(h_{i,1}(z_1), h_{i,2}(z_2),\ldots, h_{i, m}(z_m)).
$$
Then for $u\in \Sigma_n$, $\psi_u(z_1,\ldots, z_m)=(h_{u,1}(z_1), h_{u,2}(z_2),\ldots, h_{u, m}(z_m))$.  Due to  the property ({\bf P}), we see that
for $u, v\in \Sigma_n$,
\begin{equation}
\label{e-dual}
\psi_u=\psi_v \mbox{ if and only if }\phi_u=\phi_v.
\end{equation}

Let $\widetilde{K}$ be the attractor of $\Psi$ and let $A$ denote the diagonal matrix $\mbox{diag}(\beta_1^{-1},\ldots, \beta_m^{-1})$. A direct calculation shows that
for each $u\in \Sigma_n$,
\begin{equation}
\label{e-la1}
A^{-n}\psi_u(z)=z+A^{-n}\psi_u(0)=z+(t_{u,1}, t_{u,2},\ldots,  t_{u,m}),
\end{equation}
where
\begin{equation}\label{e-tij}
t_{u,j}:=\sum_{p=1}^n a_{u_p, j} \beta_j^{n-p+1},\quad j=1,\ldots, d.
\end{equation}
We claim that there exists a constant $C>0$ such that for every $n\in \N$ and   $u, v\in \Sigma_n$,
\begin{equation}
\label{e-claim}
\mbox{either  $\psi_u=\psi_v$} \quad \mbox{or}\quad  |A^{-n}\psi_u(0)-A^{-n}\psi_v(0)|\geq C   n^{-(\frac{m'}{m}-1)}.
\end{equation}

To prove this claim, we apply an idea of  Garsia used in \cite[Lemma 1.51]{Garsia1962}.  Choose a positive integer $M$ so that $Ma_{i,j}$ are algebraic integers for all $i, j$. The existence of such $M$ follows from Lemma \ref{lem-algebraic}(ii).   For $u, v\in \Sigma_n$,
by \eqref{e-tij} we have
\begin{equation}
\label{e-h2}
t_{u,j}-t_{v, j}=\sum_{p=1}^n (a_{u_p, j}-a_{v_p, j})  \beta_j^{n-p+1}=H(\lambda_j),\quad j=1,\ldots, d,
\end{equation}
where $H$ is defined as in \eqref{e-h1}.
It follows from Lemma \ref{lem-algebraic}(iv) that
$$
\prod_{j=1}^d(t_{u,j}-t_{v, j})=\prod_{j=1}^d H(\lambda_j)\in \Bbb Q.
$$
Since $M^d\prod_{j=1}^d(t_{u,j}-t_{v, j})=\prod_{j=1}^d\left(\sum_{p=1}^n (Ma_{u_p, j}-Ma_{v_p, j})  \beta_j^{n-p+1}\right)$ is an algebraic integer, we have
\begin{equation}
\label{e-la2} \prod_{j=1}^d(t_{u,j}-t_{v, j})\in M^{-d}\Bbb Z.
\end{equation}

 Now assume that  $\psi_u\neq \psi_v$.  Then by the property ({\bf P}), $h_{u, j}\neq h_{v,j}$ for all $j\in \{1,\ldots, d\}$. So by \eqref{e-la1}, $t_{u,j}\neq t_{v,j}$ for all $j\in \{1,\ldots, d\}$. Hence  $\prod_{j=1}^d(t_{u,j}-t_{v, j})\neq 0$.  It follows from \eqref{e-la2} that
\begin{equation}
\label{e-la3}
\prod_{j=1}^d|t_{u,j}-t_{v, j}|\geq M^{-d}.
\end{equation}
Meanwhile by \eqref{e-h2}, a simple calculation shows that
$\prod_{j=m+1}^d|t_{u,j}-t_{v, j}|\leq Dn^{m'-m},$
where $$D:=   \left(\max\{2|a_{i,j}|:\; 1\leq i\leq \ell, \; m+1\leq j\leq d\}\right)^{d-m}\cdot \prod_{j\geq m'+1}^d\frac{|\beta_j|}{1-|\beta_j|}.$$
Combining this with \eqref{e-la3} yields
  $\prod_{j=1}^m|t_{u,j}-t_{v, j}|\geq D^{-1} M^{-d} n^{-(m'-m)}$, which implies that
\begin{eqnarray*}
|A^{-n}\psi_u(0)-A^{-n}\psi_v(0)|&=& |(t_{u,1}-t_{v, 1}, \ldots, t_{u,m}-t_{v, m})|\\
& \geq & \max_{1\leq j\leq m}|t_{u,j}-t_{v, j}|\\
&\geq & \left(\prod_{j=1}^m|t_{u,j}-t_{v, j}|\right)^{1/m}\\
&\geq & \left(D^{-1} M^{-d} n^{-(m'-m)}\right)^{1/m}.
\end{eqnarray*}
This proves the claim \eqref{e-claim} by setting $C=(DM^d)^{-1/m}$.

Now notice that for $u\in \Sigma_n$ and $x\in \Bbb C^m$, $\psi_u(x)=A^nx+\psi_u(0)$ and so $A^{-n} \psi_u(x)=x+A^{-n}\psi_u(0)$. Hence for $u\in \Sigma_n$ and $z\in \Bbb C^m$,
\begin{equation}\label{e-h3}
\begin{split}
\psi_u(\widetilde{K})\cap A^n(B(z,1))\neq \emptyset \Longleftrightarrow & (\widetilde{K}+A^{-n}\psi_u(0))\cap B(z,1)\neq \emptyset\\
 \Longleftrightarrow&  A^{-n}\psi_u(0)\in (B(z,1)-\widetilde{K}).
\end{split}
\end{equation}
However by  \eqref{e-claim}, for any two distinct maps $\psi_u, \psi_{v}$ in the set $\{\psi_u:\; u\in \Sigma_n\}$,   $|A^{-n}\psi_u(0)-A^{-n}\psi_v(0)|>C n^{-(\frac{m'}{m}-1)}$.
As the set $(B(z,1)-\widetilde{K})$ is contained in a ball $B^*$ of radius $1+\mbox{diam}(\widetilde{K})$,  a  volume argument shows that
$(B(z,1)-\widetilde{K})$ contains at most
$$
\left(\frac{2+2\mbox{diam}(\widetilde{K})}{C n^{1-(m'/m)}/2}\right)^{2m}
$$
many points in the set $\{A^{-n}\psi_u(0):\; u\in \Sigma_n\}$; to see this, simply notice that in the Euclidean space $\mathbb{C}^m\simeq\mathbb{R}^{2m}$ the balls $B(x_i, C n^{1-(m'/m)}/2)$, with $x_i\in B^*\cap  \{A^{-n}\psi_u(0):\; u\in \Sigma_n\}$, are disjoint subsets of $2B^*$ (where $2B^*$ denotes the ball of radius two times that of $B^*$ and with the same center as $B^*$).
This together with \eqref{e-h3} yields that
\begin{eqnarray*}
\kappa_n&:=&\sup_{z\in {\Bbb C}^m}\#\{\psi_u:\; u\in \Sigma_n, \;\psi_u(\widetilde{K})\cap A^n(B(z,1))\neq \emptyset\}\\
&\leq & \left(\frac{2+2\mbox{diam}(\widetilde{K})}{C n^{1-(m'/m)}/2}\right)^{2m}=\left(\frac{4+4\mbox{diam}(\widetilde{K})}{C}\right)^{2m}  n^{2(m'-m)}.\\
\end{eqnarray*}
Hence $\lim_{n\to \infty} \frac{1}{n}\log \kappa_n=0$.  Moreover, when $m'=m$ (which happens if no conjugates of $\beta$ lie on the unit circle $\{|z|=1\}$), the sequence $(\kappa_n)$ is  bounded. This completes the proof of Proposition \ref{pro-a1}.
\end{proof}

\begin{rem}
Theorem \ref{thm-1.2} and Proposition \ref{pro-a1}   can be further used to calculate the dimension of the attractors of certain algebraic IFS. This will be exploited in a separate paper.
\end{rem}

\section{Proof of Theorem \ref{thm-1.4}}
\label{S-6}
In this section, we prove Theorem \ref{thm-1.4}. The proof is quite direct and only uses the definition of the AWSC.
\begin{proof}[Proof of Theorem \ref{thm-1.4}] Take a positive integer $q$ so that $qa_i\in \Z^d$ for all $1\leq i\leq \ell$. Let $r$ be the least integer greater than $q\mbox{diam}(K)$.  Recall that
$$
W_n=\left\{i_1\ldots i_k\in \{1,\ldots, \ell\}^k:\;k\geq 1, \; \frac{1}{|m_{i_1}\ldots m_{i_k}|}\leq 2^{-n}<\frac{1}{|m_{i_1}\ldots m_{i_{k-1}}|} \right\}.
$$
Below we prove that
\begin{equation}
\label{e-wn}
\sup_{x\in \R^d}\#\{\phi_u:\; u\in W_n,\; \phi_u(K)\cap B(x, 2^{-n}) \neq \emptyset\}< (n+1)^\ell (r+2q\max_i|m_i|+1)^d,
\end{equation}
which clearly implies that $\Phi$ satisfies the AWSC.

To show \eqref{e-wn},  notice that for each $u=u_1\ldots u_k\in W_n$, the map
$\phi_u$ is of the form
\begin{equation}
\label{e-form}
\frac{x}{m_{u_1}\cdots m_{u_k}}+\frac{{\bf t}}{qm_{u_1}\cdots m_{u_k}},
\end{equation}
where ${\bf t}\in \Z^d$ and
\begin{equation}
\label{e-7.1'}
2^n\leq |m_{u_1}\cdots m_{u_k}|<2^n\max_i|m_i|,
\end{equation}
moreover,  $m_{u_1}\cdots m_{u_k}=\prod_{i=1}^\ell m_i^{\tau_i}$ for some integers  $\tau_1,\ldots, \tau_\ell \in \{0, 1,\ldots, n\}$. It follows that the  contraction ratio of $\phi_u$, as $u$ runs over $W_n$,  can take at most $(n+1)^\ell$ different values.

Suppose on the contrary that \eqref{e-wn} does not hold. Then there exists $x\in \R^d$ and $L:=(r+2q\max_i|m_i|+1)^d$ different maps $\phi_u$ with the same  contraction ratio  such that $u\in W_n$ and
 $\phi_u(K)\cap B(x, 2^{-n})\neq \emptyset$. Let $\rho$ denote this  contraction ratio. By \eqref{e-form}, each of these maps is of the form $\rho x+(\rho/q){\bf t}$ with ${\bf t}\in \Z^d$.  Hence there exist $L$ distinct points ${\bf t}_1,\ldots, {\bf t}_L\in \Z^d$ so that
 $$
 \left(\rho K +\frac{\rho}{q}{\bf t}_i\right) \cap B(x, 2^{-n})\neq \emptyset,\qquad   i=1,\ldots, L,
 $$
 equivalently
 \begin{equation}
 \label{e-7.2'}
 {\bf t}_i\in E:=B(q\rho^{-1}x, \rho^{-1} 2^{-n}q)-qK,\qquad  i=1,\ldots, L.
 \end{equation}
 By \eqref{e-7.1'}, $2^{-n}\rho^{-1}<\max_i |m_i|$.
Hence  $E$ is contained in a box in $\R^d$ with side length $q\mbox{diam}(K)+ 2q\max_i |m_i|<r+ 2q\max_i|m_i|$.  However each such box can not contain $L=(r+ 2q\max_i|m_i|+1)^d$ different  integral points. This contradicts \eqref{e-7.2'}. Hence \eqref{e-wn} holds and we are done.
\end{proof}
\appendix

\section{The proof of Theorem \ref{thm-1.1}  and some extensions}
\label{S-A}
Throughout this section, let $\Phi=\{\phi_i\}_{i=1}^\ell$ be a dimensional regular IFS of similarities on $\R^d$ with ratios $r_1,\ldots, r_\ell$, and let ${\bf p}=(p_1, \ldots, p_\ell)$ be a probability vector with strictly positive entries. Let $\mu$ be the self-similar measure generated by $\Phi$ and~${\bf p}$.

We begin with a simple lemma.

\begin{lem}
\label{lem-A1}
Let $\eta$ be the self similar measure generated by $\Phi$ and a probability vector $\tilde{{\bf p}}=(\tilde{p}_1,\ldots, \tilde{p}_\ell)$. Then for $\eta$-a.e.~$z\in \R^d$,
$$
\overline{d}(\mu,z)\leq \frac{\sum_{i=1}^\ell \tilde{p}_i\log p_i}{\sum_{i=1}^\ell \tilde{p}_i\log r_i}.
$$
\end{lem}
\begin{proof}
Let $(\Sigma, \sigma)$ denote the one-sided full shift over the alphabet $\{1,\ldots, \ell\}$, and $\pi:\Sigma\to \R^d$ the canonical coding map associated with $\Phi$, i.e.
$$
\pi x=\lim_{n\to \infty} \phi_{x_1}\circ \cdots\circ\phi_{x_n}(0),\quad x=(x_i)_{i=1}^\infty\in \Sigma.
$$
 Let $m$ be the infinite Bernoulli product measure on $\Sigma$ generated by the weight $(\tilde{p}_1,\ldots, \tilde{p}_\ell)$, i.e.
$m([x_1\ldots x_n])=\tilde{p}_{x_1}\cdots \tilde{p}_{x_n}$ for each cylinder $[x_1\ldots x_n]$. Then $\eta=m\circ \pi^{-1}$ (\cite{Hutchinson1981}).

Take a large $R>0$ such that the attractor of $\Phi$ is contained in the ball $B(0, R)$. For any $x=(x_i)_{i=1}^\infty\in \Sigma$ and $n\in \N$, since $\phi_{x_1}\circ\cdots \circ \phi_{x_n} (B(0, R))$ is a ball of radius $r_1\cdots r_n R$ which contains the point $\pi x$, we have
$$
\phi_{x_1}\circ\cdots \circ  \phi_{x_n} (B(0, R))\subset B(\pi x, 2r_1\cdots r_n R),
$$
 so by the self-similarity  of $\mu$,
\begin{eqnarray*}
\mu( B(\pi x, 2r_1\cdots r_n R))&\geq& p_{x_1}\cdots p_{x_n} \mu\left((\phi_{x_1}\circ \cdots\circ \phi_{x_n})^{-1}B(\pi x, 2r_1\cdots r_n R) \right),\\
&\geq& p_{x_1}\cdots p_{x_n} \mu\left(B(0, R) \right)\\
&=& p_{x_1}\cdots p_{x_n}.
\end{eqnarray*}
It follows that
$$
\overline{d}(\mu, \pi x)\leq \limsup_{n\to \infty}\frac{\log (p_{x_1}\cdots p_{x_n})} {\log (r_{x_1}\cdots r_{x_n})},\quad x\in \Sigma.
$$
Applying Birkhoff's ergodic theorem to the righthand side of the above inequality yields that
$$
\overline{d}(\mu, \pi x)\leq \frac{\sum_{i=1}^\ell \tilde{p}_i\log p_i}{\sum_{i=1}^\ell \tilde{p}_i\log r_i}\quad \mbox{ for $m$-a.e.~$x\in \Sigma$.}
$$
This concludes the desired result since $\eta=m\circ \pi^{-1}$.
\end{proof}
To prove Theorem \ref{thm-1.1}, we also need the following well-known result in multifractal analysis. For a proof, see e.g.~\cite[Proposition 2.5(iv)]{LauNgai1999}.
\begin{lem}
\label{lem-A2}
Let $\nu$ be a compactly supported Borel probability measure on $\R^d$. Then for any $\beta\in \R$ and $q\geq 0$,
$$
\dim_H\{z\in \R^d:\; \underline{d}(\nu, z)\leq \beta\}\leq \beta q-\tau_\nu(q).$$
\end{lem}

Now we are ready to prove Theorem \ref{thm-1.1}.

\begin{proof}[Proof of Theorem \ref{thm-1.1}]
Recall that $T$ satisfies the equation $\sum_{i=1}^\ell p_i^q r_i^{-T(q)}=1$. According to a general result of Falconer (see \cite[Theorem 6.2]{Falconer1999}),
\begin{equation}
\label{e-a0'}
\tau_\mu(q)\geq T(q),\quad q\in [0,1].
\end{equation}
This inequality will be used later.

Taking the derivative of $T$ at $q$ gives
\begin{equation}
\label{e-a0}
T'(q)=\frac{\sum_{i=1}^\ell p_i^q r_i^{-T(q)}\log p_i}{\sum_{i=1}^\ell p_i^q r_i^{-T(q)}\log r_i}
\end{equation}
and in particular,
\begin{equation}
\label{e-a*}
T'(1)=\frac{\sum_{i=1}^\ell p_i\log p_i}{\sum_{i=1}^\ell p_i\log r_i}=\dim_S\mu.
\end{equation}
Since $\Phi$ is  assumed to be dimensional regular,  we have
\begin{equation}
\label{e-a1}
\dim_H\mu =\min \{d, T'(1)\}.
\end{equation}
Meanwhile, as a general result on self-similar measures, one always has
\begin{equation}
\label{e-SS}
\tau_\mu'(1+)=\dim_H\mu.
\end{equation}
Indeed, it is known that for every self-similar measure $\mu$, $\tau_\mu'(1+)$ equals the entropy dimension of $\mu$ (see \cite[Theorem 5.1 and Remark 5.2]{ShmerkinSolomyak2016}); but since $\mu$ is exact dimensional \cite{FengHu2009}, its entropy dimension and  Hausdorff dimension coincide \cite{Young1982}.

 In what follows we prove the theorem by considering 3 different cases:  (i) $T'(1)\geq d$; (ii) $T'(1)<d$ and $T(0)\geq -d$;  (iii) $T'(1)<d$ and $T(0)< -d$.

{\sl Case (i): $T'(1)\geq d$}.  Let $K$  denote the attractor of $\Phi$. By \eqref{e-a1} we have $\dim_H\mu=d$, which implies that $\dim_B K=d$. By \cite[Theorem 1.1]{Ngai1997},
$$
\tau_\mu'(1-)\geq \dim_H\mu=d.
$$
However since $\tau_\mu(0)=-\dim_B K=-d$ and $\tau_\mu(1)=0$, the above inequality and the concavity of $\tau_\mu$ force that
$$
\tau_\mu(q)=d(q-1) \quad \mbox{for every }q\in [0,1].
$$
This proves part (1a) of the theorem.  To show part (2) of the theorem, first notice that $\tau_\mu$ is differentable on $(0,1)$ and moreover,  $\tau_\mu'(1-)=d=\dim_H\mu$, so by \eqref{e-SS} $\tau_\mu$ is differentable at $1$ as well. Next we analyze the multifractal structure of $\mu$. Since $\mu$ is exact dimensional with  $\dim_H\mu=d$, we have
$$\dim_H\{z\in \R^d:\; d(\mu, z)=d\}\geq \dim_H\mu=d;
$$
on the other hand by Lemma \ref{lem-A2}, $$\dim_H\{z:\; d(\mu, z)=d\}\leq d\cdot 1-\tau_\mu(1)=d,$$
 so we have
$$\dim_H\{z\in \R^d:\; d(\mu, z)=d\}=d=dq-\tau_\mu(q)$$
for $q\in [0,1]$.  This verifies  part (2) of the theorem.

{\sl Case (ii):  $T'(1)<d$ and $T(0)\geq -d$}.  To prove the conclusions of  the theorem, we need to show that for each $q\in [0,1]$,

\begin{equation}
\label{A-5}
\tau_\mu(q)=T(q)\quad  \mbox{ and }
\end{equation}
\begin{equation}
\label{A-6}
\dim_H\{z\in \R^d:\; d(\mu, z)=T'(q)\}=T'(q)q-T(q).
\end{equation}

To this end, fix $q\in [0,1]$.  By the concavity of $T$,  we have
$$
T'(q)q-T(q)\leq T'(0)\cdot 0-T(0)\leq d.
$$

Set $\tilde{{\bf p}}=(\tilde{p}_1,\ldots, \tilde{p}_\ell)$ with $\tilde{p}_i=p_i^qr_i^{-T(q)}$, and let $\eta$ be the self-similar measure generated by $\Phi$ and $\tilde{{\bf p}}$.
Applying \eqref{e-a*} to $\eta$ yields
$$
\dim_S\eta=\frac{\sum_{i=1}^\ell \tilde{p}_i \log \tilde{p}_i}{\sum_{i=1}^\ell \tilde{p}_i\log r_i}=\frac{\sum_{i=1}^\ell p_i^qr_i^{-T(q)}(q\log p_i-T(q)\log r_i)}{\sum_{i=1}^\ell p_i^qr_i^{-T(q)}\log r_i}=T'(q)q-T(q)\leq d,
$$
so by the dimensional regularity of $\Phi$, $\dim_H\eta=\dim_S\eta=T'(q)q-T(q)$. Applying Lemma \ref{lem-A1} and \eqref{e-a0}, we have
\begin{equation}
\label{e-a4}
\overline{d}(\mu, z) \leq \frac{\sum_{i=1}^\ell \tilde{p}_i\log p_i}{\sum_{i=1}^\ell \tilde{p}_i\log r_i}= \frac{\sum_{i=1}^\ell p_i^qr_i^{-T(q)}\log p_i}{\sum_{i=1}^\ell p_i^qr_i^{-T(q)}\log r_i}=T'(q) \quad\mbox{ for $\eta$-a.e.~$z$}.
\end{equation}
Take a strictly increasing sequence $(\alpha_n)$ of real numbers so that $\lim_n\alpha_n=T'(q)$. If $q>0$, then by Lemma \ref{lem-A2},
\begin{eqnarray*}
\dim_H\{z\in \R^d:\; \underline{d}(\mu,z)\leq \alpha_n\} &\leq  & \alpha_n q-\tau_\mu(q)\\
&< & T'(q) q-T(q),
\end{eqnarray*}
where we have used \eqref{e-a0'} in the last inequality.  Otherwise if $q=0$,  then by Lemma \ref{lem-A2}, for every $n$ and  sufficiently small $\epsilon_n>0$ so that $T(\epsilon_n)-T(0)>\alpha_n\epsilon_n$, we have
\begin{eqnarray*}
\dim_H\{z\in \R^d:\; \underline{d}(\mu,z)\leq \alpha_n\} &\leq  & \alpha_n \epsilon_n-\tau_\mu(\epsilon_n)\\
&\leq  & \alpha_n \epsilon_n-T(\epsilon_n)\\
&< & -T(0)\\
&= &T'(q)q-T(q),
\end{eqnarray*}
producing the same inequality.
 Since $\eta$ is exact dimensional with dimension $T'(q) q-T(q)$, the above inequality implies that
$$
\eta \{z\in \R^d:\; \underline{d}(\mu,z)\leq \alpha_n\}=0 \quad \mbox{ for each }n\in \N,
$$
and so
$$
\eta \{z\in \R^d:\; \underline{d}(\mu,z)<T'(q)\}=0.
$$
This together with \eqref{e-a4} yields $d(\mu,z)=T'(q)$ for $\eta$-a.e.~$z$. Hence
$$
\dim_H\{z\in \R^d:\; d(\mu,z)=T'(q)\}\geq \dim_H\eta=T'(q) q-T(q).
$$
Meanwhile by Lemma \ref{lem-A2},
$$\dim_H\{z\in \R^d:\; d(\mu,z)=T'(q)\}\leq T'(q) q-\tau_\mu(q)\leq T'(q) q-T(q).$$
 These two equations imply immediately  \eqref{A-6} and \eqref{A-5}. From \eqref{A-5} we see that $\tau_\mu$ is differentiable on $(0,1)$ and $\tau_\mu'(1-)=T'(1)$. By \eqref{e-a1}-\eqref{e-SS}, $\tau_\mu'(1+)=\min\{d, T'(1)\}=T'(1)$, so $\tau_\mu$ is also differentiable at $1$.

{\sl Case (iii):  $T'(1)<d$ and $T(0)< -d$}. Set $$\tilde{q}=\inf\{q\in (0,1):\; T'(q)q-T(q)\leq d\}.$$
By the concavity of $T$, the function $g(q):=T'(q)q-T(q)$ is decreasing in $q$  with $g(0)=-T(0)>d$ and $g(1)=T'(1)<d$. Hence
$\tilde{q}\in (0, 1)$.

If $q\in [\tilde{q},1]$, then $T'(q) q-T(q)\leq d$;   the same argument as that in Case (ii) shows that \eqref{A-6} and \eqref{A-5} hold for $q\in [\tilde{q},1]$ and that  $\tau_\mu$ is differentiable on $(\tilde{q},1]$.

 Finally assume that  $q\in [0, \tilde{q})$. By the definition of $\tilde{q}$ we have
 $$T'(\tilde{q}) \tilde{q}-T(\tilde{q})=d=\tau_\mu'(\tilde{q}+) \tilde{q}-\tau_\mu(\tilde{q}).$$
It follows that the tangent line of the graph of $\tau_\mu$ at the point $(\tilde{q}, \tau_\mu(\tilde{q})$) crosses at the $y$-axis at $(0, -d)$. Since $\tau_\mu$ is concave and
$\tau_\mu(0)\geq -d$, $\tau_\mu$ must take the linear expression of the statement over $[0, \tilde{q})$.  Notice that the slope of this line segment equals $(d+T(\tilde{q}))/\tilde{q}=T'(\tilde{q})=\tau_\mu'(\tilde{q}+)$, it follows that $\tau_\mu$ is differentiable on $(0,\tilde{q}]$ and
$$
[\tau'_\mu(1), \tau'_\mu(0+)]=\{T'(q):\; q\in [\tilde{q}, 1]\}.
$$
 Since \eqref{A-6} holds for $q\in [\tilde{q}, 1]$, we obtain the desired result on the multifractal structure of $\mu$.
This completes the proof of the theorem.
\end{proof}

\begin{rem}
\label{rem-A1}
\begin{itemize}
\item[(i)]
For a dimensional regular IFS $\Phi$ on $\R^d$ and a given $q>1$ with $T'(q)q-T(q)\leq  d$, if we have known that $\tau_\mu(q)=T(q)$  in advance, then following the same argument as in Case (ii) of the proof of Theorem \ref{thm-1.1}, we obtain that \begin{equation}
\label{e-new1}
\dim_HE_\mu(T'(q))=T'(q)q-T(q).
\end{equation}
\item[(ii)] Applying the above argument to the case when $\Phi$ is an IFS on $\R$ satisfying  the ESC,  we  conclude   that \eqref{e-new1} holds for every $q\in (1,\infty)$ so that  $T(q)\leq q-1$. Indeed for such $q$, by Theorem \ref{thm-S}(i), $\tau_\mu(q)=\min\{q-1, T(q)\}=T(q)$; meanwhile by the concavity of $T$,   $$T'(q)q-T(q)\leq \frac{T(q)-T(1)}{q-1}q-T(q)=\frac{T(q)}{q-1}\leq 1.$$
 \end{itemize}
\end{rem}

 In  the end of this section, we present several extensions of Theorem \ref{thm-1.1}.

First rather than self-similar measures, there is an analogue of Theorem \ref{thm-1.1} for the projections of quasi-Bernoulli measures  under a stronger assumption on $\Phi$.   To be more precise,  assume
that $\Phi=\{\phi_i=r_iU_i+a_i\}_{i=1}^\ell$ is an IFS of similarities on $\R^d$ such that for any $\sigma$-invariant ergodic measure $\eta$ on $\Sigma=\{1,\ldots, \ell\}^\N$, the Hausdorff dimension of the projection of $\eta$ under the coding map $\pi$ satisfies
$$
\dim_H\eta\circ \pi^{-1}=\min\left\{d,\; \frac{h(\eta)}{\lambda(\eta)}\right\},
$$
where $h(\eta)$ stands for the measure theoretic entropy of $\eta$ and $\lambda(\eta):=\sum_{i=1}^\ell\log (1/r_i)\eta([i])$.  For instance, this assumption holds for all IFS on $\R$ satisfying the ESC \cite{JordanRapaport2019}.
Suppose that $m$ is a quasi Bernoulli measure on $\Sigma$,  in the sense that there exists a constant $C>0$ such that
$$
C^{-1}m([I])m([J])\leq m([IJ])\leq C m([I])m([J]) \mbox{ for all } I, J\in \bigcup_{n=1}^\infty\{1,\ldots, \ell\}^n.
$$
Define $f, \phi_n\in C(\Sigma)$, $n\in \N$, by
$$f(x)=-\log r_{x_1} \mbox{ and } \phi_n(x)=\log m([x_1\ldots x_n]) \; \mbox{ for } x=(x_k)_{k=1}^\infty.
$$
For $q\in \R$, let $D(q)$ be the unique value so that
$$
P\left( (D(q) S_nf+q \phi_n)_{n=1}^\infty\right)=0,
$$
where $S_nf=f+f\circ \sigma+\ldots +f\circ \sigma^{n-1}$ and  $P(\cdot)$ stands for the sub-additive pressure (see e.g. \cite[Section~2.1]{BarralFeng2013}).
Then by adapting the proof of Theorem \ref{thm-1.1} and using some ideas of the proof of \cite[Theorem 1.3 (i)]{BarralFeng2013}, we can show that the conclusions of Theorem \ref{thm-1.1} (in which $T$ is replaced by $D$)  still holds for $\mu=m\circ\pi^{-1}$.

 Secondly, Theorem \ref{thm-1.1}  and Remark \ref{rem-A1} also extend to the convolutions of certain self-similar measures on $\R$.  To see it, let $\mu_i$, $i=1,2$, be the self-similar measures on $\R$  generated by $\Phi_i=\{\rho_ix+a_{i,j}\}_{j=1}^{\ell_i}$ and ${\bf p}_i=(p_{i,j})_{j=1}^{\ell_i}$. Assume that both $\Phi_1$ and $\Phi_2$ satisfy the OSC and $\log\rho_1/\log \rho_2\not\in {\Bbb Q}$.  Then according to \cite[Theorem 2.2]{Shmerkin2019}, for every $q>1$, the $L^q$-spectrum of $\mu_1*\mu_2$ is given by
$$\tau_{\mu_1*\mu_2}(q)=\min\{q-1, T_1(q)+T_2(q)\},$$
where $T_i(q):=\log(\sum_{j=1}^{\ell_i} p_{i,j}^q) /\log \rho_i$.     Moreover, by \cite[Theorem 1.4]{Hochman-S2012},    $\mu_1*\mu_2$ is exact dimensional with dimension equal to $\min\{1, T_1'(1)+T_2'(1)\}$.  Using the these results and modifying the proof of Theorem \ref{thm-1.1} correspondingly, one can show that the conclusions of  Theorem \ref{thm-1.1} still holds (in which we replace $T$ by $T_1+T_2$), furthermore, the multifractal formalism holds for $\mu_1*\mu_2$ for those $q>1$ so that
$(T_1'(q)+T_2'(q))q-(T_1(q)+T_2(q))\leq 1$.

Finally we remark that Theorem \ref{thm-1.1}  and Remark \ref{rem-A1}  extend to  a class  of dynamically driven self-similar measures on $\R$ considered in \cite{Shmerkin2019}.  To see it, let $(\mu_x)_{x\in X}$ be the dynamically driven self-similar measures generated by a pleasant model $(X, {\bf T}, \Delta, \lambda)$  satisfying the ESC (see \cite[Section~1.5]{Shmerkin2019} for the involved definitions).  Under mild assumptions (see \cite[Theorem 1.11]{Shmerkin2019}), Shmerkin  showed that the $L^q$-spectrum of $\mu_x$, for each $x\in X$ and   $q>1$, is given by  $\tau_{\mu_x}(q)=\min\{q-1, T(q)\}$, where
$$T(q):=\frac{\int \log \|\Delta(x)\|_q^q\; d{\Bbb P}(x)}{\log \lambda}.$$
Using similar arguments, we can show that under the same assumptions as in \cite[Theorem 1.11]{Shmerkin2019}, the conclusions of Theorem \ref{thm-1.1}  (in which $\mu$ is replaced by $\mu_x$) hold  for each $x\in X$; moreover for each $q>1$ with $T'(q)q-T(q)\leq 1$, we have $\dim_HE_{\mu_x}(T'(q))=T'(q)q-T(q)$ for every $x$.
\bigskip

{\noindent \bf  Acknowledgements}.
The research of both authors was supported in part by University of Paris 13, the HKRGC GRF grants (projects CUHK14301218, CUHK14304119), the Direct Grant for Research in CUHK,  and the France/Hong Kong joint research scheme PROCORE (33160RE, F-CUHK402/14). They are grateful to Pablo Shmerkin for helpful  conversations on overlapping algebraic IFSs.

\end{document}